\documentclass{amsart}[12pt]

\usepackage{amsmath, amsthm, amscd, amsfonts,}
\usepackage{graphicx,float}
\usepackage{amssymb}
\usepackage[cmtip,arrow]{xy}
\usepackage{pb-diagram,pb-xy}
\usepackage{comment}
\DeclareMathOperator{\dom}{dom}

\DeclareMathOperator{\ran}{ran}

\def\MRB{{\mathbb{R}}}

\def\ds{\displaystyle}
\def\k{\kappa}
\def\sse{\subseteq}

\def\l{\lambda}
\def\lan{\langle}
\def\ran{\rangle}
\def\a{\alpha}
\def\om{\omega}

\def\ov{\overline}

\newtheorem{theorem}{Theorem}[section]
\newtheorem{lemma}[theorem]{Lemma}

\newtheorem{claim}[theorem]{Claim}

\numberwithin{equation}{section}

\usepackage{pb-diagram}

\def\l{\lambda}
\def\ds{\displaystyle}
\def\sse{\subseteq}

\def\lan{\langle}
\def\ran{\rangle}
\def\ov{\bar}
\def\rmark{\mbox{$\rm\bf\rule{0.06em}{1.45ex}\kern-0.05em R$}}
\def\pmark{\mbox{$\rm\bf\rule{0.06em}{1.45ex}\kern-0.05em P$}}
\def\nmark{\mbox{$\rm\bf\rule{0.06em}{1.45ex}\kern-0.05em N$}}
\def\vdash{\mbox{$\rm\| \kern-0.13em -$}}

\def\l{\lambda}
\def\ds{\displaystyle}
\def\sse{\subseteq}

\def\lan{\langle}
\def\ran{\rangle}
\def\ov{\bar}
\def\rmark{\mbox{$\rm\bf\rule{0.06em}{1.45ex}\kern-0.05em R$}}
\def\pmark{\mbox{$\rm\bf\rule{0.06em}{1.45ex}\kern-0.05em P$}}
\def\nmark{\mbox{$\rm\bf\rule{0.06em}{1.45ex}\kern-0.05em N$}}
\def\vdash{\mbox{$\rm\| \kern-0.13em -$}}
\newcommand{\lusim}[1]{\smash{\underset{\raisebox{1.2pt}[0cm][0cm]{$\sim$}}
{{#1}}}}

\title[Adding a lot of random reals by adding a few]{Adding a lot of random reals by adding a few}

\author[M. Gitik and M. Golshani ]{Moti Gitik and Mohammad Golshani }

\thanks{The second author's research has been supported by a grant from IPM (No. 96030417).}

\begin{document}

\thanks{ } \maketitle



\begin{abstract}
We study pairs $(V, V_1)$ of models of $ZFC$ such that adding $\kappa$-many random reals over $V_1$ adds $\lambda$-many random reals over $V$, for some
$\lambda > \kappa.$
\end{abstract}

\thanks{ } \maketitle

\section{Introduction}
In   \cite{g-g1} and \cite{g-g2}, we studied pairs $(V, V_1)$ of models of $ZFC$ such that adding $\kappa$-many Cohen reals over $V_1$ adds
$\lambda$-many Cohen reals  over $V$, for some
$\lambda > \kappa.$
In this paper we  prove similar results for random forcing, by producing pairs
$(V, V_1)$ of models of $ZFC$ such that adding $\kappa$-many random reals over $V_1$ adds
$\lambda$-many random reals  over $V$, where by $\kappa$-random reals over $V$ we mean a sequence  $\langle r_i: i< \kappa  \rangle $
which is $\MRB(\kappa)$-generic over $V$, and $\MRB(\kappa)$ is the usual forcing notion for adding $\kappa$-many random reals (see Section \ref{random forcing}). The proofs are more involved than those
given in  \cite{g-g1} and \cite{g-g2} for Cohen reals. This is because
random reals, in contrast to Cohen reals, may depend on $\omega$-many coordinates, instead
of finitely many as in the Cohen case. Also the proofs in  \cite{g-g1} and \cite{g-g2} were based on the fact that the product of Cohen forcing
with itself is essentially the same as Cohen forcing, while this  is not true in the case of random forcing.

\section{Random real forcing}
\label{random forcing}
In this section we briefly review random forcing and refer the reader to \cite{kunen} for more details.
Suppose $I$ is a non-empty set and consider the product measure space $2^{I \times \omega}$ with the standard product measure $\mu_I$ on it.  Let $\mathbb{B}(I)$ denote the class of Borel subsets of $2^{I \times \omega}$.
Note that the sets of the form
\[
[s]=\{x \in 2^{I \times \om}: x \upharpoonright \dom(s)=s              \},
\]
where  $s: I \times \om \to 2 $ is a finite partial function,
form a basis of open sets of $2^{I \times \om}.$

For Borel sets $S, T \in \mathbb{B}(I)$ set
\[
S \sim T \iff S \bigtriangleup T \text{~is null,~}
\]
where $S \bigtriangleup T$ denotes the symmetric difference of $S$ and $T$. The relation $\sim$ is easily seen to be an equivalence relation
on $\mathbb{B}(I).$
Then $\MRB(I)$, the forcing for adding $|I|$-many random reals, is defined as

\[
\MRB(I)= \mathbb{B}(I) / \sim.
\]
Thus elements of $\MRB(I)$ are equivalence classes $[S]$ of Borel sets modulo null sets. The order relation is defined by
\[
[S] \leq [T] \iff \mu_I(S \setminus T) =0.
\]

The following fact is standard.
\begin{lemma}
\label{chain condition lemma}
$\MRB(I)$ is c.c.c.
\end{lemma}
Using the above lemma, we can easily show that $\MRB(I)$ is in fact a complete Boolean algebra.
Let $\lusim{F}$
  be an $\MRB(I)$-name for a function from $I \times \omega$ to $2$ such that for each $i \in I, n \in \omega$ and $k < 2,$ $\parallel  \lusim{F}(i, n) =k     \parallel_{\MRB(I)} = p_k^{i, n},$ where
\[
p_k^{i, n}= [\{x \in 2^{I \times \omega}: x(i, n)=k\}].
\]
This defines $\MRB(I)$-names $\lusim{r}_i \in 2^\omega, i \in I,$ such that
\[
\parallel \forall n < \omega,~\lusim{r}_i(n)= \lusim{F}(i, n)\parallel_{\MRB(I)} =1_{\MRB(I)}= [2^{I \times \omega}].
\]

\begin{lemma}
Assume $G$ is $\MRB(I)$-generic over $V$ and for each $i \in I$ set $r_i = \lusim{r}_i[G].$ Then each $r_i \in 2^\omega$
is a new real and for $i \neq j$ in $I, r_i \neq r_j$. Further, $V[G]=V[\langle  r_i: i \in I    \rangle].$
\end{lemma}
The reals $r_i$ are called random reals. By $\kappa$-random reals over $V$ we mean a sequence  $\langle r_i: i< \kappa  \rangle $
which is $\MRB(\kappa)$-generic over $V$.

Given $b=[T] \in \MRB(I)$ and $|I|$-random reals $\langle r_i: i\in I  \rangle $ over $V$, we say  $\langle r_i: i\in I  \rangle $ extends $b$ if
\[
\forall i \in I, \forall n< \omega, \exists x \in T ( \mu_I(T \cap [x \upharpoonright \{(i, m): m<n  \}]) >0 \text{~and~} \forall m<n, ~x(i, m)=r_i(m)).
\]
This simply says that if $i$ and $n$ are given, then we can extend $b$ to some
\[
\bar b=[T \cap [x \upharpoonright \{(i, m): m<n  \}]]
\]
 such that
$\bar b$ decides $r_i \upharpoonright n.$ In fact, $\bar b \Vdash$``$\forall m<n,~\lusim{r}_i(m)  = x(i, m)$''.
Note that if $\langle r_i: i< \kappa  \rangle $ is a sequence of $|I|$-random reals generated by $G$, then
\[
G=\{ [T] \in \MRB(I):  \langle r_i: i\in I  \rangle  \text{~extends~}[T]                       \}
\]

The next lemma follows from Lemma \ref{chain condition lemma}.
\begin{lemma}
\label{characterization of random reals}
The sequence $\langle r_i: i< \kappa  \rangle $ is $\MRB(\kappa)$-generic over $V$ iff for each countable set $I \subseteq \kappa, I \in V,$
the sequence $\langle r_i: i\in I  \rangle $ is $\MRB(I)$-generic over $V$.
\end{lemma}

\section{The first general fact about adding many random reals}
\label{first fact}
In this section we prove the following theorem, which is an analogue of Theorem 2.1 from \cite{g-g1}, and use it to get some consequences.
\begin{theorem}
\label{first}
Let $V_1$
be an extension of $V$. Suppose that in $V_1:$
\begin{enumerate}
\item [$(a)$] $\kappa < \lambda$ are infinite cardinals,

\item [$(b)$] $\lambda$ is regular,

\item [$(c)$] there exists an increasing sequence $\langle \kappa_n: n < \omega  \rangle$ cofinal in $\kappa.$ In particular $cf(\kappa) = \omega,$

\item [$(d)$] there exists  an increasing (mod finite) sequence $\langle
f_{\alpha}: \alpha < \lambda \rangle$ of functions in the product $\ds\prod_{n<
\omega}(\kappa_{n+1}\setminus \kappa_n),$

\item [$(e)$] there exists a club $C \subseteq \lambda$
which avoids points of countable $V$-cofinality.
\end{enumerate}
Then adding $\kappa$-many random reals over $V_1$ produces $\lambda$-many random reals over $V$.
\end{theorem}
\begin{proof}
There are two cases to consider: $(1):$ $\lambda=\kappa^{+}$
and $(2):$ $\lambda > \kappa^{+}.$ We give a proof for the first case, as the second case can be proved similarly, using ideas from [1, Theorem 2.1] (combined with the proof of the first case given below).
We may assume, for clarity of exposition, that $\min(C)=0$.

Thus assume that $\lambda=\kappa^{+},$ and force to add
$\k$-many random reals over $V_1$. We denote them  by $\lan r_{\imath, \tau} : \imath, \tau<\k\ran$. Also let $\lan f_\a : \a<\k^+\ran\in V_1$ be an
increasing (mod finite) sequence in $\ds\prod_{n<\om}(\k_{n+1}
\setminus \k_{n})$. We define a sequence $\langle  s_{\a}: \a < \kappa^+   \rangle$
of reals as follows:



Assume $\a < \k^+$.  Let $\a^*$ and $\a^{**}$ be two
consecutive points of $C$ so that $\a^*\leq \a<\a^{**}$. Let $\lan\a_{\imath} :
\imath<\k\ran$ be some fixed enumeration of the interval
$[\a^*,\a^{**})$ with $\a_0=\a^*$. Then for some $\imath<\k$,  $\a=\a_{\imath}$.
Let $k(\imath)=\min\{k<\om : r_{\imath, \imath}(k)=1\}$. Set

\begin{center}
 $\forall n<\om$, $s_{\a}(n)=r_{f_{\a}(k(\imath)+n), f_{\a}(k(\imath)+n)}(0)$.
\end{center}

The following lemma completes the proof.

\begin{lemma}
\label{randomness 1}
 $\lan s_{\a}:\a<\k^+\ran$ is a sequence of $\k^+$-many random
reals over $V$.
\end{lemma}

\begin{proof} First, we may assume that  $\langle r_{\imath, \tau} : \imath, \tau<\k\ran$ is ${\MRB}(\k \times \k)$-generic over $V_1$.
By Lemma \ref{characterization of random reals}, it suffices to show that for any
countable set $I\sse\k^+$, $I\in V$, the sequence $\lan s_{\a}:\a\in
I\ran$ is ${\MRB}(I)$-generic over $V$. Thus it suffices to prove
the following:

$\hspace{1.5cm}$ for every $p\in {\MRB}(\k \times \k)$
and every open dense subset $D\in V$

$(*)$ $\hspace{1cm}$ of ${\MRB}(I)$,~ there is
$\ov{p}\leq p$ such~ that $\ov{p} \vdash
``  \lan \lusim{s}_{\a} : \a\in I\ran$ extends

$\hspace{1.6cm}$   some element of $D$''.

Let $p$ and $D$ be as above. For simplicity suppose that
$p=1_{\MRB(\kappa \times \k)} = [2^{(\k \times \kappa) \times \omega}]$. By $(e)$ there are only finitely many $\a^*\in
C$ such that $I\cap[\a^*,\a^{**}) \neq\emptyset$, where
$\a^{**}=\min(C\setminus(\a^*+1))$. For simplicity suppose that there are exactly two $\a^*_1<\a^*_2$ in $C$ with this property. Let
$n^*<\om$ be such that for all $n\geq n^*$, $f_{\a^*_1}(n)<
f_{\a^*_2}(n)$.

Let $b=[T_b] \in D,$ where $T_b \subseteq 2^{I \times \om}.$
For $j\in\{1,2\}$, let $\{\a_{j_l}: l < k_j \leq \omega\}$ be an
 enumeration of $I \cap [\a^*_j,\a^{**}_j).$  For $j \in \{1,2\}$ and
$l < k_j$ let $\a_{j,l}=\a_{\imath_{jl}}$ where $\imath_{jl}<\k$ is
the index of $\a_{j,l}$ in the enumeration of the interval
$[\a^*_j,\a^{**}_j)$ considered  above.

For every $j_1, j_2 \in \{ 1, 2\}, l_1 < k_{j_1}, l_2 < k_{j_2}$
and $n_1, n_2 < \omega$ set
\[
c(j_1, j_2, l_1, l_2, n_1, n_2) = \parallel   \lusim{s}_{\a_{j_1,l_1}}(n_1) \neq     \lusim{s}_{\a_{j_2,l_2}}(n_2)\parallel.
\]
\begin{claim}
\label{finiteness}
The set
$\Delta= \{ (j_1, j_2, l_1, l_2, n_1, n_2): b \leq c(j_1, j_2, l_1, l_2, n_1, n_2)               \}$
is finite. Also, $(j_1, j_2, l_1, l_2, n_1, n_2) \in \Delta$ implies $(j_2, j_1, l_2, l_1, n_2, n_1) \in \Delta.$
\end{claim}
\begin{proof}
Recall that $b=[T_b].$ By shrinking $T_b$ if necessary, we can assume that $T_b$ is closed.
Then $2^{I \times \om} \setminus T_b$ is open, so there are finite  partial functions $t_k: I \times \om \to 2 $ such that
$2^{I \times \om} \setminus T_b = \bigcup_{k<\om} [t_k]$ and for $k \neq l,~[t_k] \cap [t_l]=\emptyset$.
For each $k$ set $\Omega_k= \{t: \dom(t)=\dom(t_k)$ and $t \neq t_k        \}$. Then each $\Omega_k$ is finite and $2^{I \times \om} \setminus [t_k] = \bigcup_{t \in \Omega_k} [t]$. So
\[
T_b = \bigcap_{k<\om} (2^{I \times \om} \setminus [t_k])= \bigcap_{k<\om} (\bigcup_{t \in \Omega_k} [t]).
\]
Also, as $\mu_I(T_b)>0,$ we have
\[
\mu_I(2^{I \times \om} \setminus T_b)=\sum_{k<\om} 2^{-|t_k|} <1.
\]
Note that $\mu_I(T_b)= 1 - \sum_{k<\om} 2^{-|t_k|} >0.$
Fix an increasing sequence $\langle \eta_k: k<\om  \rangle$ of natural numbers such that
\\
$(\dag)$$\hspace{4.cm}$$\sum_{k<\om} 2^{-\eta_k} < \dfrac{1-\mu_I(2^{I \times\omega} \setminus T_b)}{1+\mu_I(2^{I \times\omega} \setminus T_b)}.$

Assume, towards a contradiction, that the set $\Delta$ is infinite.
 For each $k<\omega$ let $X_k$ be a finite subset of $\Delta$ such that:
 \begin{enumerate}
\item $(j_1, j_2, l_1, l_2, n_1, n_2) \in X_k \implies$ at least one of $(\alpha_{j_1, l_1}, n_1)$
or $(\alpha_{j_2, l_2}, n_2)$ is not in $\dom(t_k)$.
\item The set
\[
\{ (\alpha_{j_i, l_i}, n_i):   (j_1, j_2, l_1, l_2, n_1, n_2) \in X_k \text{~and~} i=1,2               \} \setminus \dom(t_k)
\]
has size $2\eta_k$.
\item $X_k$'s, for $k<\omega,$ are pairwise disjoint.
\end{enumerate}
 Set
   $$Y_k=\dom(t_k) \cup \{ (\alpha_{j_i, l_i}, n_i):   (j_1, j_2, l_1, l_2, n_1, n_2) \in X_k \text{~and~} i=1,2               \}.$$

For each $t \in \Omega_k$ let
\begin{center}
$\Lambda_{k, t} = \{t': Y_k \to 2: t' \supseteq t$ and  $\exists (j_1, j_2, l_1, l_2, n_1, n_2) \in X_k , t'(\a_{j_1,l_1}, n_1) =t'(\a_{j_2,l_2}, n_2)   \}$.
\end{center}
Note that each $t' \in \Lambda_{k, t}$ is well-defined by clause (1)
above.
Let
$$\zeta_k=|\{ (\alpha_{j_i, l_i}, n_i):   (j_1, j_2, l_1, l_2, n_1, n_2) \in X_k \text{~and~} i=1,2               \} \cap \dom(t_k)|.$$
Then note that
$2\eta_k=2\xi_k + \zeta_k,$ where $\xi_k$ is the number of those $(j_1, j_2, l_1, l_2, n_1, n_2) \in X_k $ such that both
$(\alpha_{j_1, l_1}, n_1)$
and $(\alpha_{j_2, l_2}, n_2)$ are not in $\dom(t_k)$.

Set  $\bar{T}=  \bigcap_{k<\om} (\bigcup_{t \in \Omega_k}(\bigcup_{t' \in \Lambda_{k, t}} [t'])).$
Clearly, $|\Omega_k|=2^{|t_k|}-1,$ $|\Lambda_{k, t}| =2^{2\eta_k}-2^{\xi_k}$ and for each $t' \in \Lambda_{k, t}, |t'|=|t|+2\eta_k=|t_k|+2\eta_k$, and so
\[
\mu_I(\bigcup_{t \in \Omega_k}(\bigcup_{t' \in \Lambda_{k, t}} [t'])) = \sum_{t \in \Omega_k} (\sum_{t' \in \Lambda_{k, t}} \mu_I([t'])) = (2^{|t_k|}-1)(2^{2\eta_k} - 2^{\xi_k})2^{-(|t_k|+2\eta_k)}.
\]
It follows that

$\hspace{1.6cm}$$\mu_I(2^{I \times\omega} \setminus \bar T) \leq \sum_{k< \om} (1- (2^{|t_k|}-1)(2^{2\eta_k} - 2^{\xi_k})2^{-(|t_k|+2\eta_k)} )$

$\hspace{3.7cm}$$=\sum_{k< \om} (2^{\xi_k - 2\eta_k} + 2^{-|t_k|} - 2^{\xi_k - |t_k| - 2\eta_k}  )$


$\hspace{3.7cm}$$\leq \sum_{k< \om} (2^{\xi_k - 2\eta_k} + 2^{-|t_k|} + 2^{\xi_k - |t_k| - 2\eta_k}  ) $

$\hspace{3.7cm}$$ = \sum_{k<\om}2^{\xi_k - 2\eta_k}  + \sum_{k<\om} 2^{-|t_k|} + \sum_{k<\om}2^{\xi_k - |t_k| - 2\eta_k}  $

$\hspace{3.7cm}$$ \leq \sum_{k<\om}2^{\eta_k - 2\eta_k}  + \sum_{k<\om} 2^{-|t_k|} + \sum_{k<\om}2^{\eta_k - |t_k| - 2\eta_k}$ (as $\xi_k \leq \eta_k$)

$\hspace{3.7cm}$$=\sum_{k<\om}2^{-|t_k|} + \sum_{k<\om} 2^{-\eta_k} + \sum_{k<\om}2^{-|t_k|-\eta_k} $

$\hspace{3.7cm}$$\leq \sum_{k<\om}2^{-|t_k|}+ \sum_{k<\om} 2^{-\eta_k} + (\sum_{k<\om}2^{-|t_k|})(\sum_{k<\om}2^{-\eta_k})$

$\hspace{3.7cm}$$\leq \mu_I(2^{I \times\omega} \setminus T_b)+ \sum_{k<\om} 2^{-\eta_k} + \mu_I(2^{I \times\omega} \setminus T_b)(\sum_{k<\om}2^{-\eta_k})$

$\hspace{3.7cm}$$< 1$ (by $(\dag)$).

Hence
\[
\mu_I(\bar{T}) =1 - \mu_I(2^{I \times\omega} \setminus \bar T) >0.
\]

Set $\bar{b}=[\bar{T}],$
Then $\bar{b} \in \MRB(I)$ and $\bar{b} \leq b.$ Also note that:
\begin{center}
$\forall x \in \bar{T}, \forall k< \om, \exists (j_1, j_2, l_1, l_2, n_1, n_2) \in X_k,~ x(\a_{j_1,l_1}, n_1) =x(\a_{j_2,l_2}, n_2).$
\end{center}
Let $S'$ consists of those $y \in 2^{(\k \times \kappa) \times \omega}$ such that for some $k<\om$, some  $(j_1, j_2, l_1, l_2, n_1, n_2) \in X_k$ and
some
$x \in \bar{T}$
\begin{enumerate}
\item $y(f_{\a_{j_1 l_1}}(n_1), f_{\a_{j_1 l_1}}(n_1), n_1)=x(\a_{j_1 l_1}, n_1)$.
\item $y(f_{\a_{j_2 l_2}}(n_2), f_{\a_{j_2 l_2}}(n_2), n_2)=x(\a_{j_2 l_2}, n_2)$.
\item $x(\a_{j_1 l_1}, n_1)=x(\a_{j_2 l_2}, n_2).$
\end{enumerate}
Clearly, $\mu_{\k \times \k}(S')>0$.
For each $y \in S'$ let $k_y$ denote the least $k$ as above. Similarly, let $(j^y_1, j^y_2, l^y_1, l^y_2, n^y_1, n^y_2)$ denote the least  $(j_1, j_2, l_1, l_2, n_1, n_2) \in X_k$ as above (with respect to some fixed well-ordering of $\Delta$).
For some $\bar k<\om$ and $(\bar j_1, \bar j_2, \bar l_1, \bar l_2, \bar n_1, \bar n_2) \in X_k$, the set $S''=\{y \in S': k_y=\bar k$ and $(j^y_1, j^y_2, l^y_1, l^y_2, n^y_1, n^y_2)= (\bar j_1, \bar j_2, \bar l_1, \bar l_2, \bar n_1, \bar n_2)       \}$ has positive measure.
Let
\[
\bar S=\{y \in S'':   y(\imath_{\bar j_1, \bar l_1}, \imath_{\bar j_1, \bar l_1}, 0)= y(\imath_{\bar j_2, \bar l_2}, \imath_{\bar j_2, \bar l_2}, 0)=1                \}.
\]
Then $\mu_{\k \times \k}(\bar S) =$$1 \over 4$$\mu_{\k \times \k}(S'')>0$ and if $\bar p=[\bar S],$ then $\bar p \in \MRB(\k \times \k)$ and
\[
\bar p \text{~}\Vdash \text{``} \lusim{k}(\imath_{\bar j_1, \bar l_1})=\lusim{k}(\imath_{\bar j_2, \bar l_2})=0\text{''}.
\]

For each $y \in \bar{S}$, if $x$ (with $\bar k$ and $\bar u$) is a witness as above, then

$\hspace{1.8cm}$$\bar{p} \Vdash ~\text{``} \lusim{s}_{\a_{\bar j_1, \bar l_1}}(\bar n_1)= \lusim{r}_{f_{\a_{\bar j_1, \bar l_1}}(\bar n_1),
f_{\a_{\bar j_1, \bar l_1}}(\bar n_1)}(0)$


$\hspace{4.5cm}$$= y(f_{\a_{\bar j_1, \bar l_1}}(\bar n_1), f_{\a_{\bar j_1, \bar l_1}}(\bar n_1), \bar n_1)$


$\hspace{4.5cm}$$=x(\a_{\bar j_1, \bar l_1}, \bar n_1)$ (by $(1)$)


$\hspace{4.5cm}$$=x(\a_{\bar j_2, \bar l_2}, \bar n_2)$ (by $(3)$)

$\hspace{4.5cm}$$ = y(f_{\a_{\bar j_2, \bar l_2}}(\bar n_2), f_{\a_{\bar j_2, \bar l_2}}(\bar n_2), \bar n_2)$ (by $(2)$)

$\hspace{4.5cm}$$=\lusim{r}_{f_{\a_{\bar j_2, \bar l_2}}(\bar n_2), f_{\a_{\bar j_2, \bar l_2}}(\bar n_2)}(0)$

$\hspace{4.5cm}$$=\lusim{s}_{\a_{\bar j_2, \bar l_2}}(\bar n_2)\text{''}.$

So $\bar{b} \nleq  \parallel  \lusim{s}_{\a_{\bar j_1, \bar l_1}}(\bar n_1) \neq    \lusim{s}_{\a_{\bar j_2, \bar l_2}}(\bar n_2)    \parallel,$ and since
 $\bar{b} \leq b,$ we have
$$b \nleq    \parallel  \lusim{s}_{\a_{\bar j_1, \bar l_1}}(\bar n_1) \neq    \lusim{s}_{\a_{\bar j_2, \bar l_2}}(\bar n_2)    \parallel.$$

It follows that $(\bar j_1, \bar j_2, \bar l_1, \bar l_2, \bar n_1, \bar n_2) \notin \Delta$,
which  is a  contradiction. The second part of the claim is evident and the claim follows.
\end{proof}
Say that $(j, l)$ appears in $\Delta$ if $(j, l)=(j_1, l_1)$ for some $(j_1, j_2, l_1, l_2, n_1, n_2) \in \Delta.$
Also set
\[
\Lambda=\{(j, l): (j, l) \text{~appears in~} \Delta      \}.
\]
Then $|\Lambda| \leq 2|\Delta|$ is finite.
 Let $m^*$, with  $n^* \leq m^*<\om$, be
such that for all $n\geq m^*$
all of the values
\[
f_{\a^*_1}(n), ~ f_{\a_{j_1, l_1}}(n), ~ f_{\a_{j_2, l_2}}(n), ~ f_{\a^*_2}(n),
\]
are all different, where $(j_1, l_1),(j_2, l_2) \in \Lambda.$

\begin{claim}
There exists $p_1 \leq p$ such that for all  $(j,l) \in \Lambda$,
$$p_1\Vdash `` k(\imath_{jl})=\min\{k<\om: \lusim{r}_{\imath_{jl}, \imath_{jl}}(k)=1\}=
m^*~\text{''}.$$
\end{claim}
\begin{proof}
Let  $S_{p_1} \subseteq 2^{(\k \times \kappa) \times \omega}$ be defined by
\[
S_{p_1}=\{y \in 2^{(\k \times \kappa) \times \omega}: \forall (j,l)\in \Lambda~[(\forall n<m^*,~ y(\imath_{jl}, \imath_{jl}, n)   =0) \text{~and~}  y(\imath_{jl}, \imath_{jl}, m^*)=1 ]               \}.
\]
Then $\mu_{\k \times \k}(S_{p_1})=2^{-|\Lambda|(m^*+1)} >0$, so
$p_1=[S_{p_1}] \in \MRB(\kappa \times \k)$.
Further, for all $(j, l) \in \Lambda$ and $n<m^*$,
$p_1\Vdash `` \lusim{r}_{\imath_{jl}, \imath_{jl}}(n)=0
$'',  $p_1\Vdash `` \lusim{r}_{\imath_{jl}, \imath_{jl}}(m^*)=1
$'' and thus for  $(j, l) \in \Lambda$,
$$p_1\Vdash `` k(\imath_{jl})=\min\{k<\om: \lusim{r}_{\imath_{jl}, \imath_{jl}}(k)=1\}=
m^*~\text{''},$$

as required
\end{proof}
Before we continue, let us make an assumption on $T_b.$ For each $n<\om$ let $\Phi_n=\{(\a_{\imath_{jl}}, m):  (j, l)  \in \Lambda, m<n            \} \subseteq I \times \om$.
Then for a countable subset $T'$ of $I \times \om$, $\{ x \upharpoonright \Phi_n: x \in T'\} = 2^{H_n}$, for all $n<\om$. As $[T_b]=[T_b \cup T'],$
let us assume without loss of generality that $T' \subseteq T_b.$

Set
\begin{center}
$J=\{ f_{\a_{jl}}(m^*+m): (j, l) \in \Lambda$ and $ m< \om             \} \subseteq \k.$
\end{center}
Note that by our choice of $m^*,$ for all $m$ and all $(j_1, l_1), (j_2, l_2) \in \Lambda$, $f_{\a_{j_1l_1}}(m^*+m) \neq f_{\a_{j_2l_2}}(m^*+m)$.
Set

$\hspace{1.cm}$$\bar S=\{y \in S_{p_1}: \forall n < \om, \exists x \in T_b, \forall (j, l)  \in \Lambda, \forall m<n$

$\hspace{4.2cm}$  $(~y(f_{\a_{jl}}(m^*+m), f_{\a_{jl}}(m^*+m), m)=x(\a_{\imath_{jl}}, m)~)                \}.$

By  the above remarks, $\bar S$ is well-defined.
We also have
$\bar S =\bigcap_{n<\om}S_n,$
where
\[
S_{n}= \{ y \in S_{p_1}:\exists x \in T_b,  \forall (j, l)  \in \Lambda, \forall m<n (~y(f_{\a_{jl}}(m^*+m), f_{\a_{jl}}(m^*+m), m)=x(\a_{\imath_{jl}}, m)~)           \}.
\]
Let
$$W_n=\{ (f_{\a_{jl}}(m^*+m), f_{\a_{jl}}(m^*+m), m):     (j, l)  \in \Lambda, m<n           \}$$
 and
\[
\Delta_n=\{t: W_n \to 2 \mid \exists x \in T_b,   \forall (j, l)  \in \Lambda, \forall m<n,   ( y(f_{\a_{jl}}(m^*+m), f_{\a_{jl}}(m^*+m), m)=x(\a_{\imath_{jl}}, m)   )             \}.
\]
By our assumption, $T' \subseteq T_b, |\Delta_n|=2^{|W_n|}$, and hence, $\mu_{\k \times \k}(\bigcup_{t \in \Delta_n} [t])= \sum_{t \in \Delta_n}2^{|t|}= 2^{|W_n|}2^{-|W_n|}=1$.
We have
$S_n= S_{p_1} \cap \bigcup_{t \in \Delta_n} [t]$, so
\begin{center}
$\mu_{\k \times \k}(S_n) =\mu_{\k \times \k}(S_{p_1})+\mu_{\k \times \k}( \bigcup_{t \in \Delta_n} [t]) - \mu_{\k \times \k}(S_{p_1} \cup  \bigcup_{t \in \Delta_n} [t])=\mu_{\k \times \k}(S_{p_1}).$
\end{center}

It follows that
$\mu_{\k \times \k}(S_{p_1}\setminus S) =\mu_{\k \times \k}( \bigcup_{n<\om }(S_{p_1}\setminus S_n) \leq \sum_{n<\om}\mu_{\k \times \k}(S_{p_1}\setminus S_n)=0$,
and so
 $\mu_{\k \times \k}(S)=\mu_{\k \times \k}(S_{p_1})>0$.
Let $\ov{p}=[\bar S]$. Then $\ov{p} \in \MRB(\k \times \k)$ and $\ov{p} \leq p.$
\begin{claim}
\label{main claim}
$\ov{p} \Vdash$``$ \lan \lusim{s}_{\a_{jl}} : (j,l)\in \Lambda \ran$ extends $b$''.
\end{claim}
\begin{proof}
Suppose $(j,l) \in \Lambda$ and $n<\om.$  Let $y \in \bar S$. Thus
we can find $x \in T_b$ such that
\[
\forall m<n ~(~y(f_{\a_{jl}}(m^*+m), f_{\a_{jl}}(m^*+m), m)=x(\a_{\imath_{jl}}, m)~).
\]
But then

$\hspace{1.1cm}$$\ov{p} \Vdash ~\text{``} \lusim{s}_\a(m)=\lusim{s}_{\a_{jl}}(m)$

$\hspace{3.2cm}$$= \lusim{r}_{f_{\a_{jl}}(m^*+m), f_{\a_{jl}}(m^*+m)}(0)$

$\hspace{3.2cm}$$= y(f_{\a_{jl}}
(m^*+m), f_{\a_{jl}}
(m^*+m), 0)$

$\hspace{3.2cm}$$ = x(\a_{jl},m)$

$\hspace{3.2cm}$$ =x(\a, m)    \text{''.}$

The result follows.
\end{proof}
We now consider those $(j, l)$'s, $j \in \{1, 2\}, l< k_j$, which do not appear in $\Delta$. Fix  such a pair $(j, l)$. Also let $n<\om.$ Then there
is $(j_1, l_1) \in \Lambda$  such that for each $m<n,$ $b \nleq c(j, j_1, l, l_1, m, m)$,
i.e., $b \nVdash$``$ \lusim{s}_{\a_{j,l}}(m) \neq     \lusim{s}_{\a_{j_1,l_1}}(m) $''.
So there exists $b_{jln}=[T_{jln}] \leq b$ such that $\forall m<n,~b_{jln} \Vdash$``$ \lusim{s}_{\a_{j,l}}(m) =     \lusim{s}_{\a_{j_1,l_1}}(m)$''.

Note that $\mu_I(T_{jln}\setminus T_b)=0$. Since there are only countably many such tuples $(j, l, n),$
\begin{center}
$\mu_I(\bigcup_{n<\om, (j,l)\in \Lambda}T_{jln} \setminus T_b )=0.$
\end{center}
This implies $[T_b] = [T_b \cup \bigcup_{ n<\om, (j, l)\in \Lambda}T_{jln}],$
so without loss of generality, each $T_{jln}$ is contained in $T_b$ where $n<\om$ and $(j, l) \in \Lambda.$
 Now Claim \ref{main claim} implies the following:
\begin{claim}
\label{fullness}
$\ov{p} \Vdash$``$ \lan \lusim{s}_{\a} : \a \in I \ran$ extends $b$''.
\end{claim}
$(*)$ follows, which completes the proof of Lemma \ref{randomness 1}.
\end{proof}
Theorem \ref{first} follows.
\end{proof}
The next theorem follows immediately from Theorem \ref{first} and the arguments from
\cite{g-g1}.
\begin{theorem}
\begin{itemize}
\item [(a)] Suppose that $V$ satisfies $GCH$,
$\k=\bigcup_{n<\om} \k_n$ and $\bigcup_{n<\om} o(\k_n)=\k$
(where $o(\k_n)$ is the Mitchell order of $\k_n$). Then there
exists a cardinal preserving generic extension $V_1$ of $V$
satisfying $GCH$ and having the same reals as $V$ does, so that
adding $\k$-many random reals over $V_1$ produces $\k^+$-many random
reals over $V$.

\item [(b)] Suppose $V$ is a model of $GCH$. Then there is a generic
extension $V_1$ of $V$ satisfying $GCH$ so that
the only
cardinal of $V$ which is collapsed in $V_1$ is $\aleph_1$ and such that
adding
$\aleph_\om$-many random reals to $V_1$ produces
$\aleph_{\om+1}$-many of them over $V$.

\item [(c)] Suppose $V$ satisfies $GCH$. Then there is a generic
extension $V_1$ of $V$ satisfying $GCH$ and having the same reals
as $V$ does, so that
the
only cardinals of $V$ which are collapsed in $V_1$ are $\aleph_2$
and $\aleph_3$ and
 such that adding $\aleph_\om$-many random reals to $V_1$
produces $\aleph_{\om+1}$-many of them over $V$.

\item [(d)] Suppose that $\k$ is a strong cardinal, $\l\geq\k$ is
regular and $GCH$ holds. Then there exists a cardinal preserving
generic extension $V_1$ of $V$ having the same reals as $V$ does,
so that adding $\k$-many random reals over $V_1$ produces $\l$-many
of them over $V$.

\item [(e)] Suppose that there is a strong cardinal and $GCH$ holds. Let
$\a<\om_1$. Then there is a model $V_1\supset V$ having the same
reals as $V$ and satisfying $GCH$ below $\aleph^{V_1}_{\om}$ such
that adding $\aleph^{V_1}_{\om}$-many random reals to $V_1$
produces $\aleph^{V_1}_{\a+1}$-many of them over $V$.
\end{itemize}
\end{theorem}
We can also use ideas of the proof of Theorem \ref{first} to get the following theorem,
which is an analogue of [1, Theorem 3.1] for random reals.

\begin{theorem} Suppose that $V$ satisfies $GCH$. Then there is a cofinality
preserving generic extension $V_1$ of $V$ satisfying $GCH$ so that
adding a random real over $V_1$ produces $\aleph_1$-many random
reals over $V$.
\end{theorem}

\section{The second general fact about adding many random reals}
In this section, we prove our second general result which is an analogue of Theorem 2.1 form \cite{g-g2}.
Then we use the result to obtain similar results as in \cite{g-g2} for random reals.
\begin{theorem}
\label{second}
Suppose $\kappa < \lambda$ are infinite (regular or singular) cardinals, and let $V_1$
be an  extension of $V.$ Suppose that in $V_1:$
\begin{itemize}
\item [$(a)$] $\kappa < \lambda$ are still infinite cardinals.

\item [$(b)$] there exists an increasing sequence $\langle \kappa_n: n < \omega  \rangle$ of regular cardinals, cofinal in $\kappa.$ In particular $cf(\kappa) = \omega.$

\item [$(c)$] there is an increasing (mod finite) sequence $\langle
f_{\alpha}: \alpha < \lambda \rangle$ of functions in the product
$\prod_{n< \omega}(\kappa_{n+1}\setminus\kappa_n).$

\item [$(d)$] there is a partition $\langle S_{\sigma} : \sigma<\kappa \rangle$ of $\lambda$ into sets of size $\lambda$ such
that for every countable
set $I \in V$ and  every $\sigma<\kappa$ we have $|I\cap S_{\sigma}|< \aleph_0.$
\end{itemize}
Then adding $\kappa$-many random reals over $V_1$ produces $\lambda$-many random reals over $V.$
\end{theorem}

\begin{proof}
Force to add
$\k$-many random reals over $V_1$. Let us write them as  $\lan r_{i,\sigma} : i,\sigma <\k\ran$.
Also in $V,$ split  $\kappa$  into $\kappa$-blocks $B_{\sigma}, \sigma<\kappa,$ each of size $\kappa,$ and  let $\lan f_\a : \a<\l\ran\in V_{1}$ be an
increasing (mod finite) sequence in $\prod_{n<\om}(\k_{n+1}
\setminus \k_{n})$. Let $\a<\l$. We define a real $s_\a$ as
follows. Pick $\sigma<\kappa$ such that $\a\in S_{\sigma}.$ Let $k_{\a}=min\{k<\omega: r_{\sigma, \sigma}(k) \}=1$ and set

\begin{center}
$\forall n<\om$, $s_\a(n)=r_{f_\a(n+k_{\a}),\sigma}(0)$.
\end{center}

The following lemma completes the proof.

\begin{lemma} $\lan s_\a:\a<\l\ran$ is a sequence of
$\l$-many random reals over $V$.
\end{lemma}
\begin{proof} First note that $\lan r_{i,\sigma} :
i,\sigma<\k\ran$ is ${\MRB}(\k\times\k)$-generic over $V_1$. By Lemma \ref{characterization of random reals}, it
suffices to show that for  any countable set $I\sse \l$, $I\in V$,
the sequence $\lan s_\a:\a\in I\ran$ is ${\MRB}(I)$-generic over
$V$. Thus it suffices to prove the following

$\hspace{1.5cm}$ For every $p\in {\MRB}(\k\times\k)$
and every open dense subset $D\in V$

$(*)$$\hspace{1cm}$ ~~of ${\MRB}(I)$,~ there is
$\ov{p}\leq p$ such~ that $\ov{p} \Vdash
\ulcorner   \lan \lusim{s}_\a : \a\in I\ran$ extends

$\hspace{1.5cm}$  some element of $D\urcorner. $

 Let $p$ and $D$ be as above and for simplicity suppose that
$p=1_{\MRB(\k\times\k)}=[2^{\k \times \k \times \om}]$. Let $b=[T_b]\in D,$ where $T_b \subseteq 2^{I \times \omega}$.
 As $I$ is countable, we can find  $\{\sigma_j: j< \ov{\om} \leq \omega        \} \subseteq \lambda$
 such that
 $$I= I \cap \bigcup_{\sigma<\lambda} S_\sigma = \bigcup_{j< \ov{\om}} (I \cap S_{\sigma_j}),$$
 and each $I \cap S_{\sigma_j}$ is non-empty.
 By $(d)$, each $I\cap S_{\sigma_j}$ is finite, say
 \[
 I\cap S_{\sigma_j} = \{ \a_{j, 0}, \dots, \a_{j, k_j-1}          \}.
 \]
 For every $j_1, j_2 < \ov{\om},$ $l_1 < k_{j_1}, l_2 < k_{j_2}$ and $n_1, n_2 < \om$ set
 \[
c(j_1, j_2, l_1, l_2, n_1, n_2) = \parallel   \lusim{s}_{\a_{j_1,l_1}}(n_1) \neq     \lusim{s}_{\a_{j_2,l_2}}(n_2)    \parallel.
\]
The following can be proved as in Claim \ref{finiteness}.
\begin{claim}
\label{finiteness2}
The set
$\Delta= \{ (j_1, j_2, l_1, l_2, n_1, n_2): b \leq c(j_1, j_2, l_1, l_2, n_1, n_2)               \}$
is finite. Also, $(j_1, j_2, l_1, l_2, n_1, n_2) \in \Delta$ implies $(j_2, j_1, l_2, l_1, n_2, n_1) \in \Delta.$
\end{claim}
 Let $\Lambda =\{j< \ov{\om}:$ there exists $  (j_1, j_2, l_1, l_2, n_1, n_2) \in \Delta$  with $j=j_1       \}$.
 Then $\Lambda$
 is finite. For each $j \in \Lambda,$ by $(c)$, we can find $n^*_j<\omega$ such that for all $n\geq n^*_j$ and $\a^*_1<\a^*_2$ in $I\cap S_{\sigma_j}$ we have
$f_{\a^*_1}(n)<f_{\a^*_2}(n).$

 Let
\begin{center}
$S'= [\{x \in 2^{\kappa \times \k \times \omega}:  \forall j\in \Lambda ( \forall n< n^*_j, x(\sigma_j, \sigma_j, n)=0$ and $x(\sigma_j, \sigma_j, n^*_j)=1  )    \}  ]$
\end{center}
Then $\mu_{\k \times \k}(S') = 2^{-|\Lambda|(n^*_j+1)}>0$, and so
$p'=[S'] \in  \MRB(\k \times \k).$ Also, for each $j \in \Lambda$ and $l< k_j$, $p' \Vdash \ulcorner k_{\a_{jl}}= n^*_j  \urcorner$.
Let
\[
\bar S= \{ y\ \in S': \forall n<\om \exists x \in T_b, \forall j \in \Lambda \forall l< k_j \forall m<n~ (~  y(f_{\a_{jl}}
(n^*_j+m), \sigma_j, 0)=    x(\a_{jl},m)  ~)     \}.
\]
By our choice of $n^*_j$ there are no collisions and the above definition is well-formed.
Also, by the same arguments as before, $\mu_{\k \times \k}(\bar S)=\mu_{\k \times \k}(S')>0.$

Let
$\ov{p}=[\bar S]$.
Then $\ov{p}\in \MRB(\k\times\k)$ is well-defined and for all $\a=\a_{jl} \in I$,
where
$j \in \Lambda$ and $l<k_j$, and all $y \in S_{\ov{p}}$ we can find $x \in T_b$ such that for $m<n$,

$\hspace{1.1cm}$$\ov{p} \Vdash ~\text{``} \lusim{s}_\a(m)=\lusim{s}_{\a_{jl}}(m)$

$\hspace{3.2cm}$$= \lusim{r}_{f_{\a_{jl}}(n^*_j+m), \sigma_j}(0)$

$\hspace{3.2cm}$$= y(f_{\a_{jl}}
(n^*_j+m), \sigma_j, 0)$

$\hspace{3.2cm}$$ = x(\a_{jl},m)$

$\hspace{3.2cm}$$ =x(\a, m)    \text{''.}$

This implies
\begin{center}
$\ov{p}\Vdash \ulcorner \langle \lusim{s}_{\a_{jl}}: j \in \Lambda, l< k_j   \rangle$ extends $b \urcorner.$
\end{center}
Now, as in the proof of Claim \ref{fullness}, we have the following:
\begin{claim}
$\ov{p}\Vdash \ulcorner \langle \lusim{s}_{\a}: \a \in I  \rangle$ extends $b \urcorner.$
\end{claim}
$(*)$ follows and we are done.
\end{proof}
The theorem follows.
\end{proof}
The following theorem follows from Theorem \ref{second} and the arguments from \cite{g-g2}.
\begin{theorem}
\label{second general theorem}
\begin{itemize}
\item [(a)] Suppose that $GCH$ holds in $V,$  $\kappa$ is a cardinal of countable cofinality and
there are $\kappa$ many measurable cardinals below $\kappa$.
Then there is a cardinal– preserving forcing extension $V_1$ of $V$ not adding new reals and such that adding $\k$-many random reals
 random reals over
$V_1$ produces $\k^{+}$-many random reals over $V$.

\item [(b)] Suppose that $V_1 \supseteq V$ are such that:
\begin{enumerate}
\item $V_1$ and $V$ have the same cardinals and reals,

\item $\kappa < \lambda$ are infinite cardinals of $V_1$,

\item there is no partition $\lan S_{\sigma}:\sigma<\k \ran$ of $\l$ in $V_1$ as in Theorem 3.1$(d).$
\end{enumerate}
Then adding $\kappa$-many random reals over $V_1$ cannot produce $\lambda$-many random reals over $V.$

\item [(c)] The following are equiconsistent:
\begin{enumerate}
\item There exists a pair $(V_1, V_2), V_1 \subseteq V_2$, of models of set theory with the same cardinals and reals and a cardinal $\kappa$ of cofinality $\omega$ (in $V_2$) such that adding $\kappa$-many random reals over $V_2$ adds more than $\kappa$-many random reals over $V_1.$

\item There exists a cardinal $\delta$ which is a limit of $ \delta$-many measurable cardinals.
\end{enumerate}

\item [(d)] Suppose that $V_1 \supseteq V$ are such that $V_1$ and $V$ have the same cardinals and reals and $\aleph_{\delta}$ is less than the first fixed point of the $\aleph$-function. Then adding $\aleph_{\delta}$-many random reals over $V_1$ cannot  produce $\aleph_{\delta+1}$-many random reals over $V.$

\item [(e)]  Suppose $GCH$ holds and there exists a cardinal $\kappa$ which is of cofinality $\omega$ and is a limit of $\kappa$-many
 measurable cardinals. Then there is pair $(V_1, V_2)$ of models of $ZFC, V_1 \subseteq V_2$ such that:
\begin{enumerate}
\item $V_1$ and $V_2$ have the same cardinals and reals.

\item $\kappa$ is the first fixed
point of the $\aleph$-function in $V_1$ (and hence in $V_2$).

\item Adding $\kappa$-many random reals over $V_2$ adds
$\kappa^{+}$-many random reals over $V_1.$
\end{enumerate}
\end{itemize}
\end{theorem}

Moti Gitik,
School of Mathematical Sciences, Tel Aviv University, Tel Aviv, Israel.

E-mail address: gitik@post.tau.ac.il

http://www.math.tau.ac.il/gitik/

Mohammad Golshani,
School of Mathematics, Institute for Research in Fundamental Sciences (IPM), P.O. Box:
19395-5746, Tehran-Iran.

E-mail address: golshani.m@gmail.com

http://math.ipm.ac.ir/golshani/

\end{document}